\documentclass[12 pt]{amsart}
\usepackage[english]{babel}
\usepackage[latin1]{inputenc}
\usepackage{indentfirst}
\usepackage{color}
\usepackage{graphicx}
\usepackage{newlfont}
\usepackage{soul}
\usepackage{amssymb}
\usepackage{amsmath}
\usepackage{latexsym}
\usepackage{verbatim}
\usepackage{hyperref}
\usepackage{amsthm}
\usepackage{mathrsfs}
\usepackage{pifont}
\usepackage{textcomp}
\usepackage{tikz}

\theoremstyle{plain}

\newcommand{\IR}{\mathbb{R}}

\newcommand{\IP}{\mathbb{P}}

\newcommand{\BG}{\mathscr{G}}

\newcommand{\eeq}{\mathrel{\mathop:}=}

\newcommand{\Trop}{\operatorname{Trop}}

\newcommand{\V}{\operatorname{V}}

\newcommand{\ICt}{\mathbb{C}\{\!\{t\}\!\}}
\newcommand{\op}{\operatorname}
\newcommand{\T}{\operatorname{T}}

\newtheorem{theorem}{Theorem}[section]

\newtheorem*{theorem*}{Theorem}
\newtheorem{proposition}[theorem]{Proposition}
\newtheorem*{proposition*}{Proposition}
\newtheorem*{lemma*}{Lemma}
\newtheorem{conjecture}[theorem]{Conjecture}

\newtheorem{lemma}[theorem]{Lemma}
\newtheorem{question}{Question}
\newtheorem{introteo}{Theorem}
\theoremstyle{definition}
\newtheorem{definition}[theorem]{Definition}
\newtheorem{example}[theorem]{Example}
\theoremstyle{remark}
\newtheorem{remark}[theorem]{Remark}

\title{Tropical Chow Hypersurfaces}
\author{Paolo Tripoli}
\address{Mathematics Institute, University of Warwick, Coventry CV4 7AL, United Kingdom.}
\email{P.Tripoli@warwick.ac.uk}

\begin{document}

\begin{abstract}
	Given a projective variety $X\subset\IP^n$ of codimension $k+1$ the Chow hypersurface $Z_X$ is the hypersurface of the Grassmannian $\op{Gr}(k, n)$ parametrizing projective linear spaces that intersect $X$.
	We introduce the tropical Chow hypersurface $\Trop(Z_X)$. This object only depends on the tropical variety $\Trop(X)$ and we provide an explicit way to obtain $\Trop(Z_X)$ from $\Trop(X)$. We also give a geometric description of $\Trop(Z_X)$.
	We conjecture that, as in the classical case, $\Trop(X)$ can be reconstructed from $\Trop(Z_X)$ and prove it for the case when $X$ is a curve in $\IP^3$. This suggests that tropical Chow hypersurfaces could be the key to construct a tropical Chow variety.
\end{abstract}

\maketitle

\section{Introduction}

Given a variety $X\subset \IP^{n}$ of codimension $k+1$ its associated Chow hypersurface is the hypersurface of the Grassmannian $\op{Gr}(k,n)$ of $k$-dimensional linear subspaces of $\IP^n$ defined by:
\begin{equation*}
Z_X \eeq \{L\in  \op{Gr}(k,n) \mid L\cap X \neq \emptyset      \}.
\end{equation*}
Chow hypersurfaces were first introduced by Cayley in \cite{Cayleychowform} for curves in $\IP^3$ and then generalized by Chow and Van der Waerden in \cite{Chowforms}. The main feature of the Chow hypersurface is that it is the vanishing locus of a single polynomial equation, called \emph{Chow form}, which uniquely determines the original variety $X$. In this paper we introduce a \emph{tropical Chow hypersurface} and we study some of its properties.

The Chow polytope $\op{P}_{\op{ch}_X}\subset \IR^{n}$ is the weight polytope associated to the Chow form $\op{ch}_X$. Kapranov,  Sturmfels and Zelevinsky studied in \cite{KapStuZel} the relation between the Chow polytope and the initial degenerations of $X$.

The codimension-one skeleton of the normal fan of $\op{P}_{\op{ch}_X}$ was studied by Fink in \cite{Fink} as a natural candidate for a tropical Chow hypersurface. However, unlike in the classical case, this tropical hypersurface does not determine the original variety. 

In this paper we present a different approach to tropicalize the Chow hypersurface.
We define, by analogy with the classical case, the tropical Chow hypersurface in the tropical Grassmannian $\op{TrGr}(k,n)$:
\begin{equation*}
Z_{\Trop(X)} \eeq \{\Lambda\in  \op{TrGr}(k,n) \mid \Lambda\cap \Trop(X) \neq \emptyset      \}.
\end{equation*}
If we embed $\op{Gr}(k, n)$ in the projective space $\IP^{\binom{n+1}{k+1}-1}$ via the Pl\"ucker embedding we can also construct the tropical variety $\Trop(Z_X)$. 

\begin{introteo}
We have the following equality of sets: 
\begin{equation*}
\Trop(Z_X)=Z_{\Trop{X}}.
\end{equation*}
\end{introteo}

Our main result concerns the structure of $\Trop(Z_X)$. We define $\varphi$ to be the linear map $\varphi: \IR^{n+1}/\IR\rightarrow \IR^N/\IR$ defined by $\varphi (a_0,\ldots, a_{n})=(\sum_{i\in I}a_i)_I$. Let $\psi$ be the toric morphism associated to $\varphi$. This is the monomial morphism defined by $\psi (x_0,\ldots, x_{n})=(\prod_{i\in I}x_i)_I$. We denote by $\star$ the Hadamard product and by $+$ the Minkowski sum (for their definitions see Section~\ref{sec:tropch}). We also denote by $\BG_u$ the variety $\BG_u \eeq \{L\in \op{Gr}(k,n) | [1:\ldots:1]\in L \}$ and by $\Gamma_0$ the tropical variety $\Gamma_0 \eeq \{\Lambda\in \op{TrGr}(k,n) | (0,\ldots,0)\in \Lambda \}$.

\begin{introteo}
Let $X\subset \IP^n$ be an algebraic variety that intersects the torus $\T^n$. Then we have the following equalities: 
\begin{align*}
Z_{X}=\psi(X)\star\BG_u,\\
\Trop(Z_X)=\varphi(X) + \Gamma_0.
\end{align*}
\end{introteo}


We prove that the map that sends $\Trop(X)$ to $\Trop(Z_X)$ is injective in the case that $X$ is a curve in $\IP^3$. We conjecture that this holds for any variety $X$. In particular, the argument used in \cite{Fink} to show the non-injectivity of their construction does not work here. 

In the last section we describe an application to the study of tropicalization of families of varieties.

\subsection*{Notation}

We  set throughout the paper $N=\binom{n+1}{k+1}$. The Grassmannian $\op{Gr}(k, n)$ parametrizes $k$-dimensional projective linear spaces in $\IP^n$, and it is naturally embedded in $\IP^{N-1}$  via the Pl\"ucker embedding. The coordinates of $\IP^{N-1}$ are indexed by the collection of all subsets of cardinality $k+1$ of $[n+1]=\{0, \ldots,n\}$, we denote this collection by $\binom{[n+1]}{k+1}$.

We denote by $\T^{n}$ the embedded torus $\T^{n}\eeq \{[x_0:\ldots:x_n] \mid x_0\cdot \ldots\cdot x_n\neq 0\}\subset \IP^n$. Moreover, given an algebraic variety $X\subset\IP^n$ we use the notation $X^\circ=X\cap \T^n$. 

We write $\IR^{n+1}/\IR$ for the quotient of $\IR^{n+1}$ by the one-dimensional linear space generated by the vector $(1, \ldots, 1)$. Given a variety $X\subset \IP^n$ we will denote by $\Trop(X)\subset \IR^{n+1}/\IR$ the tropicalization of $X^\circ$. By a slight abuse of notation, we also refer to $\Trop(X)$ as the tropicalization of $X$.

We work over an algebraically closed field $K$ with a non trivial valuation $\op{val}:K\setminus\{0\}\rightarrow \IR$. We denote by $\Gamma_{\op{val}} \subset \IR$ the image of $\op{val}$ and we say that a point $p\in \IR^{n+1}/\IR$ is $\Gamma_{\op{val}}$-rational if it has a representative in $ \Gamma_{\op{val}}^{n+1}$.

\subsection*{Acknowledgments}  The author thanks Diane Maclagan for helpful suggestions and close reading, Sara Lamboglia for proofreading and Bernd Sturmfels for encouraging this project. 
The author was supported by EPSRC grant EP/L505110/1.

\section{Tropical Chow Hypersurfaces}\label{sec:tropch}

In this section we begin by describing the structure of the Chow hypersurface $Z_X$ and then define a tropical Chow hypersurface.

We denote by $L_z$ the linear space corresponding to a point $z\in \op{Gr}(k,n)$. Given a variety $X\subset \mathbb{P}^{n}$ of dimension $n-k-1$ the Chow hypersurface associated to $X$ is
\begin{equation*}
Z_X\eeq \{z\in \op{Gr}(k, n) \mid X\cap L_z \neq \emptyset\}\subset \op{Gr}(k, n).
\end{equation*}
For a complete exposition of the main properties of $Z_X$ see \cite{GelKapZel}. For a quick and more concrete introduction we suggest \cite{DalStur}. If $X$ is irreducible of degree $d$, then $Z_X$ is an irreducible hypersurface of the Grassmannian $\op{Gr}(k, n)$ of degree $d$ (see, for example, {\cite[Chapter 2, Proposition 2.2]{GelKapZel}}).

The torus $\T^{n}$ naturally embeds in $\T^{N-1}$ via the following regular map:
\begin{equation}\label{defpsi}
\begin{array}{cccc}
\psi: & \T^{n} & \longrightarrow &\T^{N-1}  \\
	 &[x_0:\ldots:x_{n}]& \longmapsto & [\prod_{i\in I}x_i]_I. \\
 \end{array}
\end{equation}
For $x\in \mathbb{P}^{n}$ we denote by $\mathscr{G}_x$ the subvariety of the Grassmannian $\op{Gr}(k,n)$ defined by 
$$\mathscr{G}_x\eeq\{z\in \op{Gr}(k, n) \mid x\in L_z\}.$$

Given two quasi-projective varieties $X,Y\subset \mathbb{P}^n$, their \emph{Hadamard product} $X\star Y\subset \mathbb{P}^n$ is defined to be the closure in the Zariski topology of the set
\begin{equation*}
\{[x_0y_0, \ldots, x_{n}y_{n}] \mid [x_0, \ldots, x_{n}]\in X \text { and } [y_0, \ldots, y_{n}]\in Y     \}.
\end{equation*}
Given two points $x=[x_0, \ldots, x_{n}]$ and $y=[y_0, \ldots, y_{n}]$ in $\mathbb{P}^n$, whenever it is well defined, we also denote by $x\star y \in \mathbb{P}^n$ the point with coordinates $[x_0y_0, \ldots, x_{n}y_{n}]$. For a variety $X\subset \mathbb{P}^{n}$ we consider the set of linear spaces intersecting $X$ in the torus $\T^{n}\subset \mathbb{P}^{n}$. We define $Z_{X^\circ}\eeq \{z \in \op{Gr}(k, n) \mid  X^\circ\cap L_z \neq \emptyset\}$.

\begin{lemma}\label{hadalemma}
Let $X\subset \mathbb{P}^{n}$ be an irreducible variety of dimension $n-k-1$, and let $u=[1:\ldots :1]\in \T^{n}$. 
Then $\overline{Z}_{X^\circ}=\psi(X^\circ)\star\mathscr{G}_u$. In particular, if $X$ intersects the torus $\T^{n}$, then $Z_X=\psi(X^\circ)\star\mathscr{G}_u$.
\end{lemma}
\begin{proof}

There is a natural action of the torus $\T^n$ on the torus $\T^{N-1}$. For $y\in \T^{n}$ this action is defined by the multiplication map
\begin{equation*}
\begin{array}{cccc}
	m_y: & \T^{N-1} & \longrightarrow &\T^{N-1}  \\
			 &[p_I]_I& \longmapsto & [p_I\prod_{i\in I}y_i]_I. \\
 \end{array}
\end{equation*}
The multiplication map $m_y$ is the Hadamard product with $\psi(y)$. It can be extended to an automorphism of $\mathbb{P}^{N-1}$ with inverse $m_{y'}$, for $y'=[y_0^{-1}: \ldots: y_{n}^{-1}]$. Moreover the Pl\"ucker ideal is homogeneous with respect to the grading of $K[p_I\ \mid \ I\in \binom{[n+1]}{k+1}]$ associated to this torus action, as a consequence $m_y$ preserves the Grassmannian. 

We claim that, for $x\in \mathbb{P}^{n}$, if we restrict $m_y$ to $\mathscr{G}_x$ we get an isomorphism
\begin{equation}\label{multisoEq}
	\begin{array}{cccc}
		m_y|_{\mathscr{G}_x}: & \mathscr{G}_x & \stackrel{\sim}{\longrightarrow} &\mathscr{G}_{x\star y}.
	\end{array}
\end{equation}

To prove this we just need to show that $m_y(\mathscr{G}_x)\subset \mathscr{G}_{x\star y}$, as the claim will then follow from the analogous statement for $m_{y'}$.
Fix a point $p\in \op{Gr}(k, n)$. The coordinates $p_I$ of $p$ arise as the determinants of the maximal minors of some $k\times n$ matrix 
\begin{equation}\label{matA}
	A_p=\left( 
		\begin{array}{ccc}
			a_{0,0} & \ldots & a_{0,n} \\
			\vdots & \ddots & \vdots \\
			a_{k,0} & \ldots & a_{k,n} 
		\end{array} 
	\right)
\end{equation}
whose rows are a basis for the linear space $L_p$ corresponding to $p$.
The minor indexed by $I$ of the matrix
\begin{equation}\label{matApsi}
	A_{p\star\psi(y)}=\left( 
		\begin{array}{ccc}
			a_{0,0}y_0 & \ldots & a_{0,n}y_{n} \\
			\vdots & \ddots & \vdots \\
			a_{k,0}y_0 & \ldots & a_{k,n}y_{n} 
		\end{array} 
	\right)
\end{equation}
has determinant $p_I\prod_{i\in I} y_i = (m_y(p))_I$ which shows that the linear space corresponding to $m_y(p)\in \op{Gr}(k, n)$ is the rowspace of $A_{p\star\psi(y)}$.
By assumption $x$ lies in the linear span of the rows of $A_p$, so there exists some $m\in K^{k+1}$ such that $\sum_j m_j a_{j,i}=x_i$ for $i=0, \ldots, n$. The entries $a'_{i,j}=a_{i,j}y_i$ of $A_{p\star\psi(y)}$ satisfy $\sum_j m_j a'_{j,i}=\sum_j m_j a_{j,i}y_i=x_iy_i=(x\star y)_i$ for $i=1, \ldots, n$ proving the claim. 
 
We can now conclude the proof of the first part of the statement. A point $z\in  \op{Gr}(k, n)$ is in $Z_{X^\circ}$ if and only if there is a point $x\in X^{\circ}$ such that $x\in L_z$. The latter condition is, by definition, equivalent to $z\in \mathscr{G}_x$ so that, by Equation \eqref{multisoEq}, we can find a $y \in \mathscr{G}_u$ with $z= m_x(y)$. As $m_x(y) = \psi (x)\star y$ we have that
\begin{equation*}
\overline{Z}_{X^\circ}=\overline{\{\psi(x)\star y \in \op{Gr}(k, n) \mid x\in X^\circ \text{ and } y\in \mathscr{G}_u\}}=\psi(X^\circ)\star \mathscr{G}_u.
\end{equation*}
We now prove the last statement. We have that $Z_X$ is irreducible as $X$ is (see \cite[Chapter 3, Proposition 2.2]{GelKapZel}). Moreover $Z_{X^\circ}$ is an open subset of $Z_X$, as its complement $Z_X\setminus Z_{X^\circ}$ is the Zariski closed set $\{L \mid L\cap X\cap\V(x_0\cdot\ldots\cdot x_{n})\neq \emptyset\}$. Finally  $Z_{X^\circ}$ is not empty as $X^\circ$ is not, so $Z_X=\overline {Z}_{X^\circ}$ concluding the proof.
\end{proof}

\begin{remark}\label{os:fibrahada}
It is useful to notice that the morphism $\psi(X^{\circ})\times \mathscr{G}^{\circ}_u \rightarrow \psi(X^{\circ})\star \mathscr{G}^{\circ}_u$ defined by $(x,y)\mapsto x\star y$ is generically one-to-one. This is equivalent to the well-known fact that a generic linear space $L_z$ with $z\in Z_X$ intersects $X$ in a unique point.
\end{remark}

We now establish a tropical version of Lemma~\ref{hadalemma}. A \emph{tropical variety} $\Sigma\subset\mathbb{R}^{n+1}/\mathbb{R}$  of dimension $n-k-1$ is a balanced weighted $\op{\Gamma_{\op{val}}}$-rational polyhedral complex of pure dimension $n-k-1$. A detailed introduction to this notion can be found in \cite[Chapter 3]{Mac}. There, however, the name tropical variety is restricted to polyhedral complexes that arise as $\Trop(X)$ for some $X\subset \T^n$.

We recall the notion of \emph{Minkowski sum} for polyhedral complexes.
Given two polyhedra $\sigma_1$, $\sigma_2$ of dimension $d_1$ and $d_2$ and weight $m_{\sigma_1}$ and $m_{\sigma_2}$, their Minkowski sum is defined, as a set, to be 
\begin{equation*}
\sigma_1 + \sigma_2=\{a+b \mid a\in \sigma_1, \ b\in \sigma_2      \}   .
\end{equation*}

The weight $m_{\sigma_1 + \sigma_2}$ of $\sigma_1 + \sigma_2$ is defined to be 
\begin{equation} \label{eq:minkw}
m_{\sigma_1 + \sigma_2}=\left\{
\begin{array}{ll}
0&\text{if }\dim (\sigma_1 + \sigma_2) \neq d_1+d_2\\
m_{\sigma_1} m_{\sigma_2}[N_{\sigma_1 + \sigma_2}:N_{\sigma_1} + N_{\sigma_2}]&\text{if }\dim (\sigma_1 + \sigma_2) = d_1+d_2
\end{array}\right.
\end{equation}
where $N_{\sigma}$ denotes, for a cone $\sigma\subset \mathbb{R}^{n+1}/\mathbb{R}$, the lattice generated by the integer points of $\sigma$.

The Minkowski sum $\Sigma_1+\Sigma_2$ of two polyhedral complexes $\Sigma_1,\ \Sigma_2\subset \mathbb{R}^{n+1}/\mathbb{R}$ of pure dimension is defined as a set to be the Minkowsi sum of the underlying sets of $\Sigma_1$ and $\Sigma_2$. The set $\Sigma_1+\Sigma_2$ is actually a polyhedral complex, and we can give it a polyhedral complex structure so that, for any $\sigma_1\in \Sigma_1$ and $\sigma_2\in \Sigma_2$, the polyhedron $\sigma_1 + \sigma_2$ is a union of polyhedra of $\Sigma_1+\Sigma_2$. If $\Sigma_1$ and $\Sigma_2$ are weighted, then we define the multiplicity of a polyhedron $\sigma\in \Sigma_1+\Sigma_2$ to be 
\begin{equation*}
m_{\sigma}=\sum_{\sigma_1+\sigma_2\supset \sigma} m_{\sigma_1+\sigma_2}.
\end{equation*}

Equivalently $\Sigma_1+\Sigma_2$ is the image of $\Sigma_1\times\Sigma_2\subset \mathbb{R}^{n+1}/\mathbb{R}\times \mathbb{R}^{n+1}/\mathbb{R}$ under the map 
\begin{equation*}
\begin{array}{cccc}
\alpha: & \mathbb{R}^{n+1}/\mathbb{R}\times \mathbb{R}^{n+1}/\mathbb{R} & \longrightarrow &\mathbb{R}^{n+1}/\mathbb{R} \\
 & (a,b) & \longmapsto & a+b,
 \end{array}
\end{equation*}
where the polyhedron $\sigma_1\times\sigma_2\in\Sigma_1\times\Sigma_2$ has weight $m_{\sigma_1}m_{\sigma_2}$.

The Minkowski sum of tropical varieties is the tropical analogue of the Hadamard product of algebraic varieties in the following sense. Given two varieties $X,Y\subset \T^n$, if the map $X\times Y \rightarrow X\star Y$ is generically one-to-one, then we have (see {\cite[Theorem 5.5.11]{Mac}}):
\begin{equation}\label{hadaminko}
\Trop(X\star Y)=\Trop(X)+\Trop(Y).
\end{equation}

The Pl\"ucker embedding realizes the Grassmannian $\op{Gr}(k,n)$ as a subvariety of the projective space $\mathbb{P}^{N-1}$. This allows us to define the tropical Grassmannian $\op{TropGr}(k,n)$ as the tropicalization $\Trop(\op{Gr^\circ}(k,n))$ of the intersection of $\op{Gr}(k,n)$ with the embedded torus $\T^{N-1}\subset \mathbb{P}^{N-1}$. The tropical Grassmannian is a parameter space for tropical varieties $\Trop(L)$ where $L$ is a linear space whose Pl\"ucker coordinates are all different from $0$. 
This condition comes from the fact that we are considering the tropicalization of the intersection $\op{Gr^\circ}(k,n)$ of $\op{Gr}(k,n)$ with the torus $\T^{N-1}$. 
For the rest of the paper we refer to those tropical varieties as tropicalized linear spaces. 

Given a point $p\in \op{TrGr}(k,n)$, we denote by $\Lambda_p\subset \mathbb{R}^{n+1}/\mathbb{R}$ the tropicalized linear space corresponding to it.

Consider a $\Gamma_{\op{val}}$-rational point $p\in \mathbb{R}^{n+1}/\mathbb{R}$ and the tropicalized linear space $\Lambda_q\subset\mathbb{R}^{n+1}/\mathbb{R}$ corresponding to the point $q\in \op{TrGr}(k,n) \subset \mathbb{R}^{N}/\mathbb{R}$. We pick a point $x$ with valuation $\op{val}(x)=p$ and a linear space $L$ with $\Trop(L)=\Lambda_q$. Equation \eqref{hadaminko} implies that the translation $\{p\} + \Lambda_q$ of $\Lambda_q$ by $p$ is the tropicalization of the Hadamard product $\{x\}\star L$. 

\begin{definition}\label{defgamma0}
Given a $\Gamma_{\op{val}}$-rational point $p\in \mathbb{R}^{n+1}/\mathbb{R}$ and $x\in\T^n$ with valuation $\op{val}(x)=p$, we define $\Gamma_p\eeq \Trop(\mathscr{G}_x)$.  
\end{definition}

\begin{lemma}\label{lem:trchset} 
\leavevmode
\begin{enumerate}
\item \label{lp:1} The tropical variety $\Gamma_p$ does not depend on the choice of the point $x$.
\item \label{lp:2} Given two $\Gamma_{\op{val}}$-rational points $p, q\in \mathbb{R}^{n+1}/\mathbb{R}$ we have $\Gamma_p=\varphi(q-p)+\Gamma_q$.
\item \label{lp:3} Given a tropicalized linear space $\Lambda\subset \mathbb{R}^{n+1}/\mathbb{R}$ and $x\in \T^n$, we have $\op{val}(x)\in \Lambda$ if and only if there exists a linear space $L\subset \mathbb{P}^n$ such that $\Lambda=\Trop(L)$ and $x\in L$.
\item \label{lp:4} We have the following equality of subsets of $\mathbb{R}^{n+1}/\mathbb{R}$
\begin{equation*}
	\Gamma_p=\{q \in \op{TrGr}(k, n) \mid p\in \Lambda_q\}.
\end{equation*}
\end{enumerate}
\end{lemma}
\begin{proof}
Given $x,y\in \T^n$, we have seen in the proof of Theorem~\ref{hadalemma} that $\mathscr{G}_x=\psi(x\star y^{-1})\star \mathscr{G}_y$ where $y^{-1}=[y_0^{-1}: \ldots: y_{n}^{-1}]$. 

If $\op{val}(x)=p$ and $\op{val}(y)=q$ then $\op{val}(x\star y^{-1})=p-q$, and so, by Equation \eqref{hadaminko}, the equality $\mathscr{G}_x=\psi(x\star y^{-1})\star \mathscr{G}_y$ tropicalizes to $\Gamma_p=\varphi(p-q)+\Gamma_q$. This proves \eqref{lp:2} and, as a particular case when $p=q$, \eqref{lp:1}.

We now prove \eqref{lp:3}. One implication is trivial: if $\Lambda=\Trop(L)$ and $x\in L$ then $\op{val}(x)\in \Lambda$. On the other hand if $\Lambda=\Trop(L)$ and $\op{val}(x)\in \Lambda$ then, by the Fundamental Theorem of Tropical Geometry (\cite[Theorem 3.2.5]{Mac}), there exists $y\in L$ with $\op{val}(y)=\op{val}(x)$. We have that $L'=(x\star y^{-1})\star L$ is again a linear space and $x\in L'$. Moreover $\op{val}(x\star y^{-1})=(0,\ldots, 0)$ and therefore $\Trop(L')=0+\Trop(L)=\Lambda$.

To conclude, \eqref{lp:4} is a consequence of \eqref{lp:3} as, for any fixed $x$ with $\op{val}(x)=p$, we have
\begin{equation*}
\Gamma_p=\Trop(\mathscr{G}_x)=\{q \in \op{TrGr}(k, n)\ \mid \Lambda_q=\Trop(L_z)\text{ and } z\in \mathscr{G}_x\}.
\end{equation*}
\end{proof}

\begin{remark}
Given $p\in \mathbb{R}^{n+1}/\mathbb{R}$ not necessarily $\Gamma_{\op{val}}$-rational, we can define $\Gamma_p$ to be the weighted polyhedral complex $\varphi(p)+\Gamma_0$. This is consistent with Definition \ref{defgamma0}.
\end{remark}

One consequence of Lemma~\ref{lem:trchset}, and in particular of point \eqref{lp:3}, is that the support set of $\Trop(Z_X)$ can be described in terms of $\Trop(X)$. Take a $\Gamma_{\op{val}}$-rational  point $p\in \Trop(X)$ and suppose that a tropicalized linear space $\Lambda$ intersects $\Trop(X)$ at $p$. Then there is a point $x\in X$ with valuation $\op{val}(x)=p$. Using Lemma~\ref{lem:trchset} we see that there exists a linear space $L\subset \mathbb{P}^n$, with $\Lambda = \Trop(L)$, that contains $x$. This shows the following equality of sets:
\begin{equation}\label{seteq1}
	\left|\Trop(Z_X) \right| = \{ p \in \op{TrGr}(k,n) \mid \Lambda_p \cap \Trop(X) \neq \emptyset \}.
\end{equation}

For any subset $S\subset \mathbb{R}^n/\mathbb{R}$, a tropicalized linear space $\Lambda_p$ is intersects $S$ if and only if $p\in \Gamma_s$ for some $s\in S$. By Lemma~\ref{lem:trchset}, $\Gamma_s=\varphi(s) + \Gamma_0$, and we get another equality of sets:
\begin{equation}\label{seteq2}
	\{ p \in \op{TrGr}(k,n) \mid \Lambda_p \cap S \neq \emptyset \} = \varphi(S) + \left| \Gamma_0 \right|.
\end{equation}

Combining Equation \eqref{seteq1} and Equation \eqref{seteq2} we get that 
\begin{equation*}
\left|\Trop(Z_X) \right| = \left| \varphi(\Trop(X))\right| + \left| \Gamma_0 \right|.
\end{equation*}
This is actually not just an equality of sets, but an equality of tropical varieties as the following Theorem shows.

\begin{theorem}\label{troptochmul}
Let $X$ be of pure dimension $n-k-1$ and assume none of its irreducible components is contained in $\V(x_1\cdot\ldots\cdot x_{n})$. We have the following equality of tropical varieties: 
\begin{equation*}
	\Trop(Z_X)=\varphi(\Trop(X)) + \Gamma_0.
\end{equation*}
\end{theorem}
\begin{proof}
Let $u=[1:\ldots: 1]\in \T^n$ and let $\psi$ be the map defined by Equation \eqref{defpsi}.

As $\psi$ is a monomial morphism we have that, see \cite[Corollary 2.6.10]{Mac}, $\Trop(\psi(X^\circ))$ equals $\Trop(\psi) \ (\Trop(X)) \subset \mathbb{R}^{N}/\mathbb{R}$ where $\Trop(\psi)$ is the linear map $\Trop(\psi):\mathbb{R}^{n+1}/\mathbb{R}\rightarrow \mathbb{R}^N/\mathbb{R}$ given by multiplication by the matrix $A=(a)_{i, I}$ with 
\begin{equation*}
	(a)_{i, I}=\left \{  
	\begin{array}{ll}
		0 & \text{if }x\notin I\\
		1 & \text{if }x\in I
	\end{array}
	\right.
\end{equation*}

We have $\Trop(\psi)=\varphi$ and hence $\Trop(\psi)(\Trop(X))=\varphi(\Trop(X))$.

The result follows immediately from Lemma~\ref{hadalemma}, Remark~\ref{os:fibrahada} and Equation \eqref{hadaminko}.
\end{proof}

Theorem~\ref{troptochmul} allow us to define a notion of associated hypersurface for \emph{any} tropical variety of pure dimension. We remind that by tropical variety we mean balanced weighted $\Gamma_{\op{val}}$-rational polyhedral complex, in the sense of \cite[Theorem 3.3.5]{Mac}.
\begin {definition}
Given a tropical variety $\Sigma\subset \mathbb{R}^{n+1}/\mathbb{R}$, the tropical Chow hypersurface $Z_\Sigma\subset \mathbb{R}^{N}/\mathbb{R}$ associated to $\Sigma$ is defined to be
\begin{equation*}
Z_\Sigma=\varphi(\Sigma)+\Gamma_0.
\end{equation*}
\end{definition}

Let $V\subset \mathbb{R}^{n+1}/\mathbb{R}$ be the support of a pure-dimensional fan. We denote by $\op{A}_{\text{unbal}}^{k}(V)$ the set of pure codimension-$k$ $\mathbb{Q}$-weighted $\Gamma_{\op{val}}$-rational polyhedral complexes whose support is contained in $V$, and by $\op{A}^{k}(V)$ the set of pure codimension-$k$ balanced $\mathbb{Q}$-weighted $\Gamma_{\op{val}}$-rational polyhedral complexes whose support is contained in $V$. The support set of a weighted polyhedral complex is the union of all maximal polyhedra that have non-zero multiplicity. As usual in tropical intersection theory, two weighted polyhedral complexes $(\Sigma_1, m_1)$ and $(\Sigma_2, m_2)$ are identified if their polyhedral complexes $\Sigma_1$ and $\Sigma_2$ have the same support set, and if, moreover, given two maximal dimensional polyhedra $\sigma_1\in\Sigma_1$ and $\sigma_2\in\Sigma_2$ whose relative interiors intersect, their multiplicity $m_1(\sigma_1)$ and $m_2(\sigma_2)$ are the same. In other words we identify weighted polyhedral complexes that are the same up to the choice of polyhedral structure. This also means that we can freely add and remove polyhedra with multiplicity $0$ without changing the polyhedral complex, and that a polyhedral complex is $0$ if and only if all its polyhedra have multiplicity $0$.

The set $\op{A}_{\text{unbal}}^{k}(V)$ is a vector space over $\mathbb{Q}$ (this is the main reason to consider rational weights rather than integer). 
The sum $(\Sigma, m)$ of two weighted polyhedral complex $(\Sigma_1, m_1)$ and $(\Sigma_2, m_2)$ is defined as follows. The support set of $\Sigma$ is $\Sigma_1 \cup \Sigma_2$. Up to subdividing $\Sigma_1$ and $\Sigma_2$ it can be given a polyhedral complex structure such that every polyhedron $\sigma \in \Sigma$ is a polyhedron of $\Sigma_1$ or $\Sigma_2$ (or both). The multiplicity $m(\sigma)$ of a polyhedron $\sigma\in \Sigma$ is defined to be $m_1(\sigma)+m_2(\sigma)$, where $m_1(\sigma)$ (resp. $m_2(\sigma)$) is defined to be $0$ if $\sigma$ is not a polyhedron of $\Sigma_1$ (resp. $\Sigma_2$). 
For $x\in \mathbb{Q}$, the  multiplication of $(\Sigma, m)$ by $k$ is the weighted polyhedral complex $(\Sigma, k\cdot m)$, where $k\cdot m$ is the weight defined by $(k\cdot m) (\sigma)= k\cdot m(\sigma)$ for any maximal polyhedron $\sigma\in \Sigma$.

As every polyhedral complex is a finite union of polyhedra we have that, using the operations just defined on $\op{A}_{\text{unbal}}^{k}(V)$, any weighted polyhedral complex can be written as sum of weighted polyhedral complex whose support is a single polyhedron. Moreover, as the scalar multiplication acts as multiplication on the weight, the set of polyhedral complexes made of a single codimension $k$ polyhedron, with multiplicity one is a set of generators for $\op{A}_{\text{unbal}}^{k}(V)$.

We also have a vector space structure on $\op{A}^{k}(V)$, inherited from the structure on $\op{A}_{\text{unbal}}^{k}(V)$. Actually $\op{A}^{k}(V)$ is a vector subspace of $\op{A}_{\text{unbal}}^{k}(V)$.

We define the map 
\begin{equation}\label{chowmap}
\begin{array}{cccc}
\op{Z}: & \op{A}^k(\mathbb{R}^{n+1}/\mathbb{R})& \longrightarrow & \op{A}^1(\op{TrGr}(k,n)) \\
 & \Sigma & \longmapsto & Z_\Sigma=\varphi(\Sigma)+\Gamma_0.
\end{array}
\end{equation}
We now prove that this is a linear transformation by showing that its extension $\op{Z}': \op{A}_{\text{unbal}}^k(\mathbb{R}^{n+1}/\mathbb{R}) \rightarrow  \op{A}_{\text{unbal}}^1(\op{TrGr}(k,n))$ is linear. The linearity of $\op{Z}'$ on fan that consists of a single cone follows immediately from the linearity in $m_1$ of Equation \eqref{eq:minkw}. As single polyhedra span $\op{A}_{\text{unbal}}^k(\mathbb{R}^{n+1}/\mathbb{R})$, $\op{Z}'$ is linear and then so is $\op{Z}$.

\section{From tropical Chow hypersurface to tropical variety}

In this section we address the question whether it is possible to recover the tropical variety $\Trop(X)$ from the tropical Chow hypersurface $\Trop(Z_X)$ or, in other words whether the map \eqref{chowmap} is injective. 

The hypersurface $Z_X$ is the vanishing locus in the Grassmannian of a single polynomial in the ring of polynomials in the Pl\"ucker variables $K[p_I \mid   I\in \binom{[n+1]}{k+1}]$. The coordinate ring of the Grassmannian $K[\op{Gr}(k,n)]$ is the quotient of $K[p_I]$ by the Pl\"ucker ideal. The class $\op{ch}_X$ of this polynomial in $K[\op{Gr}(k,n)]$ is uniquely determined by $X$, and it is called \emph{Chow form} of $X$.

Different lifts of $\op{ch}(X)$ to $K[p_I\mid I\in \binom{[n+1]}{k+1}]$ have different Newton polytopes in $\mathbb{R}^N/\mathbb{R}$, so that there is not a natural notion of Newton polytope for the Chow form $\op{ch}_X$. There is, however, a polytope in $\op{P_{ch_X}}\subset\mathbb{R}^{n+1}/\mathbb{R}$, called \emph{Chow polytope}. This is the weight polytope associated to the natural action of $\T^n$ on $K[\op{Gr}(k,n)]$. Explicitly, given a monomial $\prod{p_{I}^{a_I}}\in K[p_I]$, its weight is $\sum a_Ie_{I}\in \mathbb{R}^{n+1}/\mathbb{R}$, where $e_I=\sum_{i\in I} e_i$ and $e_0, \ldots, e_n$ is the image of the standard basis of $\mathbb{R}^{n+1}$ in $\mathbb{R}^{n+1}/\mathbb{R}$. For any lift $\overline{\op{ch}}_X\in K[p_I]$ of the Chow form $\op{ch}_X$ we can write $\overline{\op{ch}}_X = c_1 + \ldots +c_l$, where each $c_i$ is a sum of monomials with the same weight $p_i$. The Chow polytope $P_{\op{ch}_{X}}$ is the convex hull of the weights $p_i$, for every $i$ such that $[c_i]\neq 0$.

\begin{example}
Consider the conic $C=\V(t,x^2+y^2+z^2)\subset \mathbb{P}^3$, where $x,y,z,t$ are the coordinates of $\mathbb{P}^3$. Its Chow form can be computed (for example using the algorithm described in \cite[Section 3.1]{DalStur}) as the class in $K[\op{Gr}(k,n)]$ of the polynomial $c= p_{03}^2+p_{13}^2+p_{23}^2$. The Chow polytope is, in this case, the convex hull of the three weights $2(e_0+e_3), \ 2(e_1+e_3), \ 2(e_2+e_3)$ of the three monomials of $c$. Consider the polynomial $c'=c+p_{12}p_{03}-p_{02}p_{13}+p_{01}p_{23}$, we have again that the class of $c'$ is the Chow form of $C$ because $c$ and $c'$ only differ by an element of the Pl\"ucker ideal. The weight $e_0+e_1+e_2+e_3$ now appears as the weight of some monomial of $c'$. This is not source of ambiguity in the definition of the Chow polytope: if we sum all the monomials of $c'$ with weight $e_0+e_1+e_2+e_3$ we get the polynomial $p_{12}p_{03}-p_{02}p_{13}+p_{01}p_{23}$ whose class is $0$ modulo the Pl\"ucker ideal.
\end{example}

\begin{remark}\label{choweight}
Let $\overline{\op{ch}}_X $ be a lift of $\op{ch}_X $, and write $\overline{\op{ch}}_X = c_1 + \ldots c_l$, with each  $c_i$ being a sum of monomials with the same weight. We can always get another lift  
\begin{equation*}
	\overline{\op{ch}}_X - \sum_{i \mid [c_i]=0} c_i,
\end{equation*}
where $[c_i]$ denotes the class of $c_i$ modulo the Pl\"ucker ideal, such that the Chow Polytope $P_X$ is the convex hull of the weights of its monomials. In other words, $P_X$ is the projection of the Newton polygon of this lift under the linear map that sends the vector $e_I \in \mathbb{R}^N/\mathbb{R}$ to $\sum_{i\in I} e_i \in \mathbb{R}^{n+1}/\mathbb{R}$.
\end{remark}

The codimension-one skeleton $\mathcal{N}^1(\op{P}_X)$ of the (inner) normal fan of $\op{P}_X$ was studied by Fink in \cite{Fink}. In particular he proved a Minkowski sum decomposition for $\mathcal{N}^1(\op{P}_X)$ which we now recall.

The linear space $\Lambda_0$ corresponding to the origin $0\in \op{TropGr}(k, n)$ is called the \emph{standard} tropical linear $k$-plane. The rays of $\Lambda_0$ are $\op{pos}( e_0 ), \ldots \op{pos}( e_n )$, and every subset of $k$ rays spans a maximal cone of multiplicity one of $\Lambda_0$.

Denote by $-\Lambda_0$ the image of $\Lambda_0$ under the map $-\op{id}_{\mathbb{R}^{n+1}/\mathbb{R}}: p \mapsto -p$. Then we have (see \cite[Theorem 5.1]{Fink})
\begin{equation}\label{finkthm}
	\mathcal{N}^1(\op{P}_X)= X+ (-\Lambda_0).
\end{equation}

\begin{proposition} Let $H=\varphi(\mathbb{R}^{n+1}/\mathbb{R})$ be the image of $\mathbb{R}^{n+1}/\mathbb{R}$ in $\mathbb{R}^N/\mathbb{R}$. Then 
	\begin{enumerate}
		\item \label{it:ZXvsNP} $\Trop(Z_X)\cap H = \varphi(\mathcal{N}^1(\op{P}_X))$,
		\item \label{it:G0vsL0}$(\varphi(X) + \Gamma_0) \cap H = \varphi(X + (-\Lambda_0))$.
	\end{enumerate}
	In particular Equation~\eqref{finkthm} can be obtained by intersecting the equality of Theorem~\ref{troptochmul} with $H$.
\end{proposition}
\begin{proof}
	Fix a point $z_0\in \op{Gr}(k,n)$ with $\op{val}(z_0)=0$. We consider the morphism
	\begin{align*}
		j_{z_0}	:	\T^n 	\rightarrow	\op{Gr}(k,n)	\\
							x			\mapsto	\psi(x)\star z_0.
	\end{align*}
		We choose a lift $\overline{\op{ch}}_X\in K[p_I \mid I \in \binom{[n+1]}{k+1}]$ of $\op{ch}_X$ with the property of Remark~\ref{choweight}.
		Let $A_{z_0}$ be the subvariety of $\T^n$ defined by $A_{z_0}= \{ x\in \T^n\ \mid \ j_{z_0}(x)\in Z_X\}$. A point $j_{z_0}(x)$ is in $Z_X$ if and only if $\op{\overline{ch}}_X(j_{z_0}(x))=0$. Denote by $F$ the polynomial defined by $F(x)\eeq  \op{\overline{ch}}_X(j_{z_0}(x))$. We have $A_{z_0}=\V(F)\subset \T^n$. We claim that $\op{Newt}(F)=\op{P}_X$, so that $\Trop(A_{z_0})= \mathcal{N}^1(\op{P}_X)$. Indeed, the monomials of $F$ are obtained, up to coefficient, from the monomials of $\op{\overline{ch}}_X$ by the ring homomorphism defined by $p_I\mapsto \prod_{i\in I}x_i$. In particular the degree of a monomial of $F$ equals the weight of the corresponding monomial of $\op{\overline{ch}}_X$. The claim follows from Remark~\ref{choweight}.
		
		We can now prove \ref{it:ZXvsNP} by showing that the two sets have the same $\Gamma_{\op{val}}$-rational points. Let $v\in \mathcal{N}^1(P_X)$ be a $\Gamma_{\op{val}}$-rational point and fix any point $z_0\in \op{Gr}(k,n)$ with $\op{val}(z_0)=0$. Then there exists some $x\in A_{z_0}$ with $\op{val}(x)=v$, in particular $j_{z_0}(x) \in Z_X$ and, as $\op{val}(j_{z_0}(x))= \op{val}(\psi(x) \star z_0) = \varphi(v) + 0 =\varphi(v)$, we have $\varphi(v)\in \Trop(Z_X)$. Conversely given a $\Gamma_{\op{val}}$-rational point $\varphi(v)\in \Trop(Z_X) \cap H$ we can find $z\in Z_X$ with $\op{val}(z)=\varphi(v)$. Let $x\in \T^n$ be any point with valuation $v$, and let $z_0= \psi(x)^{-1} \star z$, where by $\psi(x)^{-1}$ we mean the point in $\mathbb{P}^{N-1}$ whose $i$-th coordinate is the inverse of the $i$-th coordinate of $\psi(x)$. Then $x\in A_{z_0}$ as $j_{z_0}(x) = z \in Z_X$, so that $v = \op{val}(x) \in \mathcal{N}^1(\op{P}_X)$. This concludes the proof of \ref{it:ZXvsNP}.
		
	A point $\varphi(v)\in \varphi(\mathbb{R}^{n+1}/\mathbb{R})$ lies in $\Gamma_0$ if and only if the linear space $\Lambda_{\varphi(v)}=v+\Lambda_0$ contains the origin, and this happens if and only if $-v\in \Lambda_0$. This implies that $\varphi(-\Lambda_0)=\Gamma_0\cap H$. On the other hand $\varphi(X)$ is already contained in $H$. Therefore $(\varphi(X) + \Gamma_0) \cap H =  \varphi(X) + (\Gamma_0 \cap H) = \varphi(X) + \varphi(-\Lambda_0) = \varphi( X + (-\Lambda_0) )$. This completes the proof of \ref{it:G0vsL0}.
 \end{proof}

The tropical variety $\Trop(X)$ is contained in $\mathcal{N}^1(\op{P}_X)$. Unfortunately, it is impossible to recover $\Trop(X)$ from $\mathcal{N}^1(\op{P}_X)$. In \cite{Fink} Fink gave an example of two distinct tropical surfaces $\Sigma_1$, $\Sigma_2$ in $\mathbb{R}^5/\mathbb{R}$, such that $\Sigma_1+(-\Lambda_0) = \Sigma_2+(-\Lambda_0)$. They are depicted in Figure~\ref{fig:alexample}. 
Computing $Z_{\Sigma_1}=\varphi(\Sigma_1)+\Gamma_0$ and $Z_{\Sigma_2}=\varphi(\Sigma_2)+\Gamma_0$, for example with the package \texttt{Polyhedra} in \texttt{Macaulay2} (\cite{M2}), we see that $Z_{\Sigma_1}\neq Z_{\Sigma_2}$. 
For example, if we give where $\mathbb{R}^{10}/\mathbb{R}$ homogeneous coordinates 
$(p_{01},p_{02}, p_{12}, p_{03}, p_{13}, p_{23}, p_{04}, p_{14}, p_{24}, p_{34})$, for
\begin{equation*}
p_1=(14, 6, 8, 11, 13, 20, 18, 16, 8, 13),\ 
p_2=(17, 12, 11, 14, 13, 23, 24, 19, 14, 16),
\end{equation*}
we have that $p_1\in Z_{\Sigma_1}$ but $p_1\notin Z_{\Sigma_2}$, and  $p_2\in Z_{\Sigma_1}$ but $p_2\notin Z_{\Sigma_2}$. 
Equivalently, consider the tropical line $\Lambda_1$ depicted in Figure~\ref{fig:linealex}. 
Let $L\subset \mathbb{P}^4$ be an algebraic line such that $\Lambda_1= \Trop(L)$ and consider the points $u=(10,8,0,5,13)$ and $v=(6,8,15,21,8)$ in $\Lambda_1$. By the Fundamental Theorem of Tropical Geometry (\cite[Theorem 3.2.5]{Mac}), there exist $x, y \in L$ such that $\op{val}(x)= u, \op{val}(y)= v$. As $L$ is spanned by $x,y$, the Pl\"ucker coordinates of $L$ are $q_{ij} = x_iy_j-x_jy_i$, for $0\leq j < i \leq 4$. The valuation $\op{val}(q_{ij})$ of $q_{ij}$ satisfies $\op{val}(q_{ij}) = \min \{\op{val}(x_iy_j), \op{val}(x_jy_i)\}= \min \{u_i+v_j, u_j+v_i\}$ because we have $u_i+v_j \neq u_j+v_i$ for all $0\leq j < i \leq 4$. This allows us to compute that $\Lambda_1$ is the tropical line corresponding to the point $p_1$. Similarly, the tropical line $\Lambda_2= \{(3,0,3,0,3)\}+\Lambda_1$ that is the translation of $\Lambda_1$ by $(3,0,3,0,3)$ corresponds to the point $p_2 = p_1 + \varphi(3,0,3,0,3) = p_1 + (3,6,3,3,0,3,6,3,6,3)$.
One can check that $\Lambda_1$ intersects $\Sigma_1$ at the point $(6,8,6,11,8)$ but it does not intersect $\Sigma_2$, while $\Lambda_2$ intersects $\Sigma_2$ at the point $(9,8,9,11,11)$ but it does not intersect $\Sigma_1$.
This fact suggests the following conjecture.

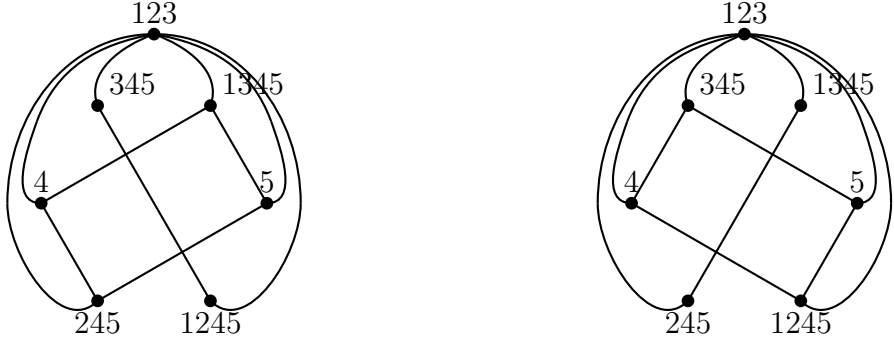
\begin{figure}
\begin{center}
	\begin{minipage}[t]{0.4\textwidth}
	\begin{center}
		\begin{tikzpicture}[style=thick]
			\draw [fill=black](0:1.5cm) circle (2pt) node[above] {$5$};  
			\draw [fill=black](60:1.5cm) circle (2pt) node[above right] {$1345$};  
			\draw [fill=black](120:1.5cm) circle (2pt) node[above right] {$345$};
			\draw [fill=black](180:1.5cm) circle (2pt) node[above] {$4$};
			\draw [fill=black](240:1.5cm) circle (2pt) node[below] {$245$};
			\draw [fill=black](300:1.5cm) circle (2pt) node[below] {$1245$};
			\draw [fill=black](0,2.25cm) circle (2pt) node[above] {$123$};
			\draw (300:1.5cm) -- (120:1.5cm);
			\draw (0:1.5cm) -- (60:1.5cm) -- (180:1.5cm) -- (240:1.5cm) -- cycle;		
			\draw    (0,2.25cm) to[out=-20,in=70] (60:1.5cm);
			\draw    (0,2.25cm) to[out=200,in=110] (120:1.5cm);		
			\draw    (0,2.25cm) to[out=-10,in=110] (35:1.9cm);
			\draw    (35:1.9cm) to[out=290,in=0] (0:1.5cm);		
			\draw    (0,2.25cm) to[out=190,in=70] (145:1.9cm);
			\draw    (145:1.9cm) to[out=250,in=180] (180:1.5cm);
			\draw    (0,2.25cm) to[out=180,in=90] (180:1.95cm);
			\draw    (180:1.95cm) to[out=270,in=225] (240:1.5cm);		
			\draw    (0,2.25cm) to[out=0,in=90] (0:1.95cm);
			\draw    (0:1.95cm) to[out=270,in=-45] (300:1.5cm);
	\end{tikzpicture}
	\end{center}
	\end{minipage}
	\hspace{0.1\textwidth}
	\begin{minipage}[t]{0.4\textwidth}
	\begin{center}
		\begin{tikzpicture}[style=thick]
			\draw [fill=black](0:1.5cm) circle (2pt) node[above] {$5$};  
			\draw [fill=black](60:1.5cm) circle (2pt) node[above right] {$1345$};  
			\draw [fill=black](120:1.5cm) circle (2pt) node[above right] {$345$};
			\draw [fill=black](180:1.5cm) circle (2pt) node[above] {$4$};
			\draw [fill=black](240:1.5cm) circle (2pt) node[below] {$245$};
			\draw [fill=black](300:1.5cm) circle (2pt) node[below] {$1245$};
			\draw [fill=black](0,2.25cm) circle (2pt) node[above] {$123$};
			\draw (60:1.5cm) -- (240:1.5cm);
			\draw (0:1.5cm) -- (120:1.5cm) -- (180:1.5cm) -- (300:1.5cm) -- cycle;		
			\draw    (0,2.25cm) to[out=-20,in=70] (60:1.5cm);
			\draw    (0,2.25cm) to[out=200,in=110] (120:1.5cm);		
			\draw    (0,2.25cm) to[out=-10,in=110] (35:1.9cm);
			\draw    (35:1.9cm) to[out=290,in=0] (0:1.5cm);		
			\draw    (0,2.25cm) to[out=190,in=70] (145:1.9cm);
			\draw    (145:1.9cm) to[out=250,in=180] (180:1.5cm);
			\draw    (0,2.25cm) to[out=180,in=90] (180:1.95cm);
			\draw    (180:1.95cm) to[out=270,in=225] (240:1.5cm);		
			\draw    (0,2.25cm) to[out=0,in=90] (0:1.95cm);
			\draw    (0:1.95cm) to[out=270,in=-45] (300:1.5cm);
	\end{tikzpicture}
	\end{center}
	\end{minipage}
\caption{The tropical varieties of $\Sigma_1$ and $\Sigma_2$ described in \cite{Fink}. A point labeled by $i_1\ldots i_k$ represents a ray generated by $e_{i_1}+\ldots+e_{i_k}$. Edges represent $2$-dimensional cones.}
\label{fig:alexample}
\end{center}
\end{figure}

\begin{figure}
\begin{center}
	\begin{tikzpicture}[style=thick]
		\draw (0,0cm) -- (2,0cm);
		\draw (2,0cm) -- (3,1cm);
		\draw (-1,-1cm) -- (0,0cm);
		\draw (-1,1cm) -- (0,0cm);
		\draw (2,0cm) -- (2,-1cm);
		\draw (3,1cm) -- (4,1cm);
		\draw (3,1cm) -- (3,2cm);
		\draw [fill=black](0,0cm) circle (2pt) node[left] {$(6,8,15,20,8)$};  
		\draw [fill=black](2,0cm) circle (2pt) node[right] {$(6,8,0,5,8)$};
		\draw [fill=black](3,1cm) circle (2pt) node[left] {$(10,8,0,5,12)$};
		
		\draw [fill=none](-1,1cm) circle (0pt) node[left] {$e_2$};  
		\draw [fill=none](-1,-1cm) circle (0pt) node[left] {$e_3$};
		\draw [fill=none](2,-1cm) circle (0pt) node[below] {$e_1$};
		\draw [fill=none](4,1cm) circle (0pt) node[right] {$e_0$};  
		\draw [fill=none](3,2cm) circle (0pt) node[above] {$e_4$};
	\end{tikzpicture}
\end{center}
\caption{The tropical line $\Lambda_1$ corresponding to the point $p_1= (14, 6, 8, 11, 13, 20, 18, 16, 8, 13)\in \op{TrGr}(1,4)$.}
\label{fig:linealex}
\end{figure}
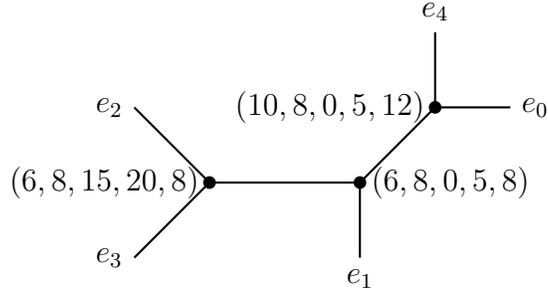

\begin{conjecture}\label{coinj}
Let $\Sigma_1, \Sigma_2 \subset \mathbb{R}^{n+1}/\mathbb{R}$ be pure $(n-k-1)$-dimensional tropical varieties. Suppose $Z_{\Sigma_1}=Z_{\Sigma_2}$ then $\Sigma_1=\Sigma_2$.
\end{conjecture}

The conjecture is false if we state it in term of equalities of sets rather than tropical varieties. The following example shows two different tropical varieties $\Sigma_1$ and $\Sigma_2$ with different support sets $\left|\Sigma_1\right| \neq \left| \Sigma_2 \right|$, such that $Z_{\Sigma_1}$ and $Z_{\Sigma_2}$ have the same support set and only differ in the multiplicities. 

\begin{example}\label{eg:contro}
Let $\Sigma_1$ be the tropical curve in $\mathbb{R}^4/\mathbb{R}$ with rays 
\begin{equation*}
\begin{array}{lll}
	\{\rho_1= \op{pos}( 1, -1, -1, 1 ), & \rho_2= \op{pos}( 1, -1, 1, -1 ), & \rho_3= \op{pos}( -1, -1, 1, 1 ),  \\
	  \rho_4= \op{pos}( -1, 1, 0, 0 ), & \rho_5= \op{pos}( 0, 1, -1, 0 ), & \rho_6= \op{pos}( 0, 1, 0, -1 ), \\
		\rho_7= \op{pos}( 0, -1, 0, 0 ), & \rho_8= \op{pos}( 1, 1, 0, 0 ), & \rho_9= \op{pos}( 0, 1, 1, 1 ) \}, 
\end{array}
\end{equation*}
and multiplicities
\begin{multline*}
	\{m_1= 1,
		m_2= 1,
		m_3= 1,
		m_4= 1,
		m_5= 1,
		m_6= 1,
		m_7= 1,
		m_8= 1,
		m_9= 1
	\}.
\end{multline*}
Let $\Sigma_2$ be the tropical curve with rays 
\begin{align*}
	\{\rho_1,
		\rho_2,
		\rho_3,
		\rho_4,
		\rho_5,
		\rho_6,
		\rho_6,
		\rho_7,
		\rho_8,
		\rho_9,
		\rho_{10}= \op{pos}( 0,1,0,0 )
	\},
\end{align*}
and multiplicities 
\begin{multline*}
	\{m_1= 1,
		m_2= 1,
		m_3= 1,
		m_4= 1,
		m_5= 1,
		m_6= 1,
		m_7= 2,
		m_8= 1,
		m_9= 1,
		m_{10}= 1
	\}.
\end{multline*}
Then the support sets of $Z_{\Sigma_1}=\varphi(\Sigma_1) + \Gamma_0$ and $Z_{\Sigma_2}=\varphi(\Sigma_2) + \Gamma_0$ are equal, despite $\Sigma_1$ and $\Sigma_2$ have different support.
\end{example}

\begin{remark}
The difference in the support of $\Sigma_1$ and $\Sigma_2$ is in the ray $\op{pos}(0,1,0,0)$. This is not a coincidence and we will show in Remark~\ref{os:setinj} that, in the case of curves in $\mathbb{R}^4/\mathbb{R}$, the support of $Z_\Sigma$ determines the support of $\Sigma$ outside $\Lambda_0 \cup (-\Lambda_0)$.
\end{remark}

In classical Algebraic Geometry the variety $X$ can be deduced from $Z_X$ as a set via the equality
\begin{equation*}
X=\{ x\in \mathbb{P}^n\mid \mathscr{G}_x\subset Z_X\}.
\end{equation*}
Example~\ref{eg:contro} also shows that this does not happen in Tropical Geometry. Set $p=(0,1,0,0)$. We have that $p\in \Sigma_2$, so that $\Gamma_p \subset \left| Z_{\Sigma_2}\right|= \left| Z_{\Sigma_1} \right|$. However $p\notin \Sigma_1$.

\subsection{Curves in space}

In the last part of this section we prove Conjecture~\ref{coinj} for space curves whose support is a fan. This property is not uncommon: if $K$ is a field extension of a field with trivial valuation $\textbf{k}$, then all the tropicalizations $\Trop(X)$ of varieties $X$ defined over $\textbf{k}$ are supported on a fan. This is the case, for example, of varieties defined over $\mathbb{C}$, when we consider the field $K=\ICt$.

We denote by $e_0, e_1,e_2, e_3$ the images of the vectors of the standard basis of $\mathbb{R}^4$ in $\mathbb{R}^4/\mathbb{R}$. The Grassmannian $\op{TrGr}(1,3)$ is embedded in $\mathbb{R}^6/\mathbb{R}$. We denote by $e_{01},e_{02},e_{12},e_{03},e_{13}, e_{23}$ the images of the standard basis of $\mathbb{R}^6$, and finally we write $f_i=\varphi(e_i)=\sum_j e_{i,j}$ for $i=0,1,2, 3$.
We denote by $\Lambda_0$ the tropical standard line in $\mathbb{R}^4/\mathbb{R}$. This is the one-dimensional fan with rays 
\begin{equation*}
	\{ \op{pos}( e_0 ), \op{pos}( e_1 ), \op{pos}( e_2 ), \op{pos}( e_3 ) \}.
\end{equation*} 
We denote by $\Lambda$ the set $\Lambda = \Lambda_0 \cup (-\Lambda_0)$.

The tropical variety $\Gamma_0\subset \op{TrGr}(1,3)$ is depicted in Figure~\ref{gamma013}. It can be computed, for example with \texttt{gfan} (\cite{gfan}) as the tropicalization of the variety $\mathscr{G}_u\subset \op{Gr}(1,3)$ of lines through the origin:
\begin{multline*}
\mathscr{G}_u=\V(p_{03}p_{12}-p_{02}p_{13}+p_{01}p_{23}, p_{12}-p_{13}+p_{23}, p_{02}-p_{03}+p_{23}, p_{01}-p_{03}+p_{13},\\ p_{01}-p_{02}+p_{12}).
\end{multline*}

\begin{figure}\label{gamma013}
\begin{center}
	\begin{tikzpicture}[style=thick]
		\draw (18:2cm) -- (90:2cm) -- (162:2cm) -- (234:2cm) --
		(306:2cm) -- cycle;
		\draw (18:1cm) -- (162:1cm) -- (306:1cm) -- (90:1cm) --
		(234:1cm) -- cycle;
		\draw (18:1cm) -- (18:2cm);
		\draw [fill=black](18:2cm) circle (2pt) node[above right] {$e_{23}$};  
		\draw [fill=black](18:1cm) circle (2pt) node[above] {$-f_0$};
		\draw (90:1cm) -- (90:2cm);
		\draw [fill=black](90:2cm) circle (2pt) node[above] {$e_{01}$};  
		\draw [fill=black](90:1cm) circle (2pt) node[right] {$-f_3$};
		\draw (162:1cm) -- (162:2cm);
		\draw [fill=black](162:2cm) circle (2pt) node[above left]  {$-f_2$};  
		\draw [fill=black](162:1cm) circle (2pt) node[above] {$e_{13}$};
		\draw (234:1cm) -- (234:2cm);
		\draw [fill=black](234:2cm) circle (2pt) node[below left] {$e_{03}$};  
		\draw [fill=black](234:1cm) circle (2pt) node[left] {$e_{12}$};
		\draw (306:1cm) -- (306:2cm);
		\draw [fill=black](306:2cm) circle (2pt) node[below right] {$-f_1$};  
		\draw [fill=black](306:1cm) circle (2pt) node[right] {$e_{02}$};
	\end{tikzpicture}
\end{center}
\caption{The tropical variety $\Gamma_0$ in $\mathbb{R}^6/\mathbb{R}$. Points in the picture represent rays in $\Gamma_0$ and edges in the picture represent $2$-dimensional cones in $\Gamma_0$. Each point is labeled with a generator of the corresponding ray. All maximal cones have multiplicity one. }
\label{fig:gamma0}
\end{figure}
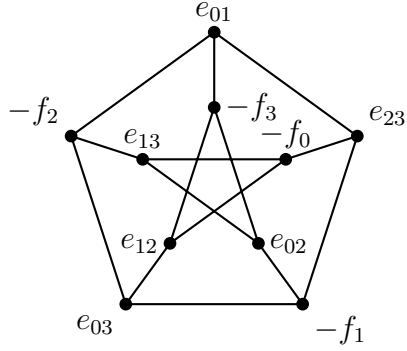

\begin{lemma}\label{lembadlocus}
	Let $H_{ij}$ be the plane in $\mathbb{R}^4/\mathbb{R}$ generated by $e_i$ and $e_j$, $0 \leq i < j \leq 3$. Then
	\begin{equation*}
		\Lambda = (H_{01}  \cup H_{23}) \cap (H_{02}  \cup H_{13}) \cap (H_{03}  \cup H_{12}).
	\end{equation*}
	\end{lemma}
\begin{proof}
	We have $H_{ij} \cap H_{kl}= \op{span}(e_i+e_j)= \op{span}(e_k+e_l)$ if $\{i,j\}$ and $\{k,l\}$ are disjoint, and $H_{ij} \cap H_{ik} = \op{span} {e_i}$ otherwise, moreover any two such lines only intersect in the origin. If we expand the right hand side we obtain
	\begin{equation*}
		(0) \cup \op{span}(e_0) \cup \op{span}(e_1) \cup \op{span}(e_2) \cup \op{span}(e_3) = \Lambda.
	\end{equation*}
\end{proof}

As in the last part of Section~\ref{sec:tropch}, we denote by $\op{A}^{0}(\Lambda)$ the set of one-dimensional balanced $\mathbb{Q}$-weighted $\Gamma_{\op{val}}$-rational polyhedral complexes in $\Lambda$ and by $\op{A}^1(\op{TrGr}(k,n))$ the set of pure codimension-one balanced $\mathbb{Q}$-weighted $\Gamma_{\op{val}}$-rational polyhedral complexes in $\op{TrGr}(1,3)$. These are vector spaces over $\mathbb{Q}$, and we have a linear map
\begin{equation}\label{ZZETA}
\begin{array}{cccc}
\op{Z}: & \op{A}^0(\Lambda)& \longrightarrow & \op{A}^1(\op{TrGr}(1,3)) \\
 & \Sigma & \longmapsto & Z_\Sigma=\varphi(\Sigma)+\Gamma_0.
\end{array}
\end{equation}

\begin{lemma}\label{injcurvlem}
	The linear map \eqref{ZZETA} is injective.
\end{lemma}

\begin{proof}
The space $\op{\hat{A}}_{\text{unbal}}^0(\Lambda)$ of (possibly not balanced) $\Gamma_{\op{val}}$-rational fans in $\Lambda$ of dimension one is an eight-dimensional vector space, whose elements are given by a choice of weights $m(\rho)$ on each of the eight rays $\rho$ of $\Lambda$. As every balanced polyhedral complex contained in $\Lambda$ is a fan, $\op{\hat{A}}_{\text{unbal}}^0(\Lambda)$ contains $\op{A}^0(\Lambda)$ as a subspace.
We denote the rays of $\Lambda$ as $\rho_1=\op{pos}(e_0),\ \rho_2=\op{pos}(e_1), \ldots, \rho_8=\op{pos}(-e_4)$.  
A natural basis of $\op{\hat{A}}_{\text{unbal}}^0(\Lambda)$ is given by the eight fans $F_1, \ldots, F_8$, where $F_i$ is the fan that has weight one on $\rho_i$ and $0$ on every other ray $\rho_j$, $j\neq i$. An element $\Sigma \in \op{\hat{A}}_{\text{unbal}}^0(\Lambda)$ can be written as $\Sigma= \sum a_i F_i$, where $a_i$ is the multiplicity of the ray $\rho_i$ in $\Sigma$.
	We denote by $\op{A}^1_{\text{unbal}}(\op{TrGr}(1,3))$ the space of (possibly not balanced) $\mathbb{Q}$-weighted $\Gamma_{\op{val}}$-rational polyhedral complexes of dimension $3$ contained in $\op{TrGr}(1,3)$. The vector space $\op{A}^1_{\text{unbal}}(\op{TrGr}(1,3))$ contains $\op{A}^1(\op{TrGr}(1,3))$ as a subspace. We can extend the map \eqref{ZZETA} to the following linear map,
	\begin{equation*}
		\begin{array}{cccc}
			\op{Z}': & \op{\hat{A}}^0_{\text{unbal}}(\Lambda)& \longrightarrow & \op{A}^1_{\text{unbal}}(\op{TrGr}(1,3))\\
							&	\Sigma													& \longmapsto & \varphi(\Sigma)+ \Gamma_0.
		\end{array}
	\end{equation*}
	We claim that $\op{Z}'$ is injective. By computing all the relevant Minkowski sums, we can see that for every fan $F_i$ in the given basis of $\op{\hat{A}}_{\text{unbal}}^0(\Lambda)$ we can find a cone $\sigma_i \in \Gamma_0$ such that $\varphi(\rho_i) + \sigma_i$ is a three-dimensional cone whose interior does not intersect the support of any other fan $\varphi(\rho_j) + \Gamma_0$ for $j\neq i$. For example we can choose 
	\begin{equation*}
	\sigma_1=\sigma_2=\sigma_5=\sigma_6=\op{pos}(e_{01}, -f_3), \  \sigma_3=\sigma_4=\sigma_7=\sigma_8=\op{pos}(e_{01}, -f_2).
	\end{equation*}
	Now consider an element $\Sigma = \sum a_i F_i\in \op{\hat{A}}_{\text{unbal}}^0(\Lambda)$, and suppose that $\op{Z}'(\Sigma) = 0$. For $i=1,\ldots 8$, the multiplicity of $\varphi(\rho_i)+ \sigma_i$ in $\op{Z}'(\Sigma)$ is given by Formula \eqref{eq:minkw} as $a_i[N_{\rho_i + \sigma_i} : N_{\rho_i} + N_{\sigma_i}]$, as every cone in $\Gamma_0$ has multiplicity one. As $\op{Z}'(\Sigma) = 0$ the multiplicity of $\varphi(\rho_i)+ \sigma_i$ in $\op{Z}'(\Sigma)$ must be $0$, and this forces $a_i=0$ as desired.
	\end{proof}

\begin{figure}
\begin{center}
	\begin{tikzpicture}[style=thick]
		\draw (1,1) -- (0,0);
		\draw (1,1cm) -- (0,2cm);
		\draw (5,1cm) -- (6,0cm);
		\draw (5,1cm) -- (6,2cm);
		\draw (5,1cm) -- (1,1cm);
		\draw [fill=black](1,1cm) circle (2pt) node[left] {$v+(a,a,0,0)$};  
		\draw [fill=black](3.6,1cm) circle (2pt) node[above] {$v$};
		\draw [fill=black](5,1cm) circle (2pt) node[right] {$v+(0,0,b,b)$};
		\draw [fill=black](0,0cm) circle (0pt) node[below left] {$e_0$};
		\draw [fill=black](0,2cm) circle (0pt) node[above left] {$e_1$};
		\draw [fill=black](6,0cm) circle (0pt) node[below right] {$e_2$};
		\draw [fill=black](6,2cm) circle (0pt) node[above right] {$e_3$};
	\end{tikzpicture}
\end{center}
\caption{The tropical line corresponding to the point $\varphi(v)+ a e_{01} + b e_{23}$.}
\label{fig:lineab}
\end{figure}
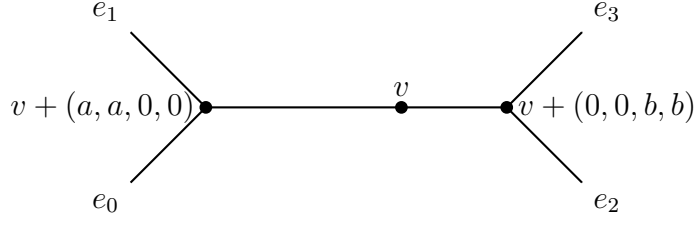

\begin{theorem}\label{curveinj}
	Let $\Sigma, \Sigma'\in \mathbb{R}^4/\mathbb{R}$ be tropical curves whose support is a fan. Suppose $Z_{\Sigma}=Z_{\Sigma'}$. Then $\Sigma=\Sigma'$.
	\end{theorem}
\begin{proof}

	We may assume that $\Sigma$ and $\Sigma'$ have the same set of rays $\{\rho_1, \ldots, \rho_m\}$ by allowing multiplicities to be $0$. For any $i$ we denote by $m_i$ the multiplicity of $\rho_i$ in $\Sigma$ and by $m'_i$ the multiplicity of $\rho_i$ in $\Sigma'$. We need to prove $m_i=m_i'$ for any $i$.

		Let $\rho_i=\op{pos}( v )$ be a ray of $\Sigma$ and $\Sigma'$. We need to show that $\rho_i$ has the same multiplicity in $\Sigma$ and $\Sigma'$. It will be enough to prove it in the case when $v\notin \Lambda$. Suppose that $\Sigma$ and $\Sigma'$ have the same multiplicities on all rays not contained in $\Lambda$, and consider the fan $\Sigma-\Sigma'$: this is the balanced weighted fan $\{\rho_1, \ldots, \rho_m\}$ where the ray $\rho_i$ has multiplicity $m_i-m_i'$. All the rays not contained in $\Lambda$ have multiplicity $0$ in $\Sigma-\Sigma'$, so that $\Sigma-\Sigma'$ is an element of $\op{A}^0(\Lambda)$. In particular $\Sigma=\Sigma'$ now follows from Lemma~\ref{injcurvlem}. Thus we can assume $v\notin \Lambda$.
		
		We claim that it will be enough to find a cone $\sigma\in\Gamma_0$ and a full dimensional subcone $\tau \subset \varphi(\rho_i) + \sigma$, such that for any ray $\rho \in \Sigma$ and for any cone $\sigma'\in\Gamma_0$ with $(\rho, \sigma') \neq (\rho_i, \sigma)$, the cone $\varphi(\rho) +\sigma'$ does not intersect $\tau$ in full dimension. Given such $\sigma$ and such $\tau$, we could find a full dimensional subcone $\tau'$ of $\tau$ that is, for some choice of fan structure, a cone of $Z_\Sigma$ and, moreover, does not intersect any other $\varphi(\rho) +\sigma'$. The multiplicity of this $\tau'$ in $Z_{\Sigma}$ would be computed by Formula  \eqref{eq:minkw} as $m_i[N_{\rho_i + \sigma_0} : N_{\rho_i} + N_{\sigma_0}]$. As  $Z_{\Sigma}=Z_{\Sigma'}$ this number would equal $m'_i[N_{\rho_i + \sigma_0} : N_{\rho_i} + N_{\sigma_0}]$, giving $m_i=m_i'$ as desired.

 We assumed that $v\notin \Lambda$ and $\Lambda$ equals, by Lemma~\ref{lembadlocus}, $(H_{01}  \cup H_{23}) \cap (H_{02}  \cup H_{13}) \cap (H_{03}  \cup H_{12})$. Without loss of generality we may assume $v\notin H_{01}  \cup H_{23}$. 
	Let $\sigma$ be the cone of $\Gamma_0$ with rays generated by $e_{01}$ and $e_{23}$, and let $\tau_{ab}=\op{pos}(\varphi(v), \varphi(v)+ae_{01}, \varphi(v)+be_{23})$, for $a,b>0$. The cone $\tau_{ab}$ is a full dimensional subcone of $\varphi(\rho_i)+\sigma = \op{pos}(\varphi(v), e_{01}, e_{23})$. We claim that, for some positive $\alpha$ and $\beta$, the cones $\sigma$ and $\tau_{\alpha\beta}$ have the desired property. In other words, we need to prove that, for $\sigma'\in\Gamma_0$, $\rho\in \Sigma$, the dimension of $\tau_{\alpha\beta} \cap (\varphi(\rho)+\sigma')$ is less than $3=\dim \tau_{\alpha\beta}$, whenever $(\sigma, \rho)\neq (\sigma', \rho_i)$.
	
	We first prove that, for $\sigma'\in \Gamma_0$ with $\sigma \neq \sigma'$, the dimension of $(\varphi(\rho_i)+\sigma) \cap (\varphi(\rho_i)+\sigma')$ is at most two and so, a fortiori, so is the dimension of $\tau_{\alpha\beta} \cap (\varphi(\rho_i)+\sigma')$ for any $\alpha$ and $\beta$. A necessary condition for $\dim ((\varphi(\rho_i)+\sigma) \cap (\varphi(\rho_i)+\sigma')) = 3$ is that $\op{span}(\varphi(\rho_i)+\sigma) = \op{span} (\varphi(\rho_i)+\sigma')$, which implies that $\op{span}(\varphi(\mathbb{R}^4/\mathbb{R})+\sigma) = \op{span} (\varphi(\mathbb{R}^4/\mathbb{R})+\sigma')$. The fan $\Gamma_0$ is depicted in Figure~\ref{fig:gamma0}, and we can check that this last condition is only satisfied by four cones: $\sigma_0=\op{pos}(e_{01},-f_2)$, $\sigma_1=\op{pos}(e_{01},-f_3)$, $\sigma_2=\op{pos}(e_{23},-f_0)$ and $\sigma_3=\op{pos}(e_{23},-f_1)$. Let us now consider the cone $\sigma_0$. It is enough to show that $\op{span}(\varphi(\rho_i)+\sigma_0) \neq \op{span}(\varphi(\rho_i)+\sigma)$, and equivalently we can show that the dimension of $\op{span}(\varphi(\rho_i)+\sigma_0) + \op{span}(\varphi(\rho_i)+\sigma)= \op{span}(\varphi(v), e_{01}, e_{23}, -f_2)$ is at least four. A simple computation in $\mathbb{R}^6$ is now enough. The matrix with rows $e_{01}$, $e_{23}$, $\varphi(v_0,v_1, v_2, v_3)$,  $-f_2$, $(1,1,1,1,1,1)$ has rank less than five if and only if $v_0 - v_1= 0$, which is false, as we assumed $v\notin H_{01}$. Similar computations work for the cones $\sigma_1$, $\sigma_2$, $\sigma_3$.
		
	We are left to prove that, for some $\alpha$ and $\beta$, $\tau_{\alpha\beta}$ intersects $\varphi(\rho)+\sigma'$ in dimension at most two for any $\rho_i\neq \rho\in \Sigma$ and any $\sigma\in\Gamma_0$. Equivalently we have to show that, for any $\rho_i\neq \rho\in \Sigma$, the fan $\tau_{\alpha \beta} \cap (\varphi(\rho) + \Gamma_0)$ has no cone of dimension more than two. 
	
	Let $P_{\alpha\beta}$ be the triangle with vertices $\varphi(v), \varphi(v)+\alpha e_{01}, \varphi(v)+\beta e_{23}$. As $\tau_{\alpha \beta}$ is the cone over $P_{\alpha\beta}$, it will be enough to show that every $\varphi(\rho) + \Gamma_0$ intersects $P_{\alpha\beta}$ in dimension at most one. 
	
	By Equation \eqref{seteq2} the fan $\varphi(\rho) + \Gamma_0$ parametrizes tropical lines intersecting the ray $\rho$. As a result a point $v+ae_{01}+be_{23}\in P_{\alpha\beta}$ is in $\varphi(\rho) + \Gamma_0$ if and only if the line $\Lambda_{ab}$ associated to it intersects $\rho$. We thus need to prove that, for some $\alpha$ and $\beta$, the set of points $\varphi(v)+ae_{01}+be_{23}$, with $0\leq a<\alpha$, $0\leq b<\beta$ and such that the line $\Lambda_{ab}$ intersect a ray $\rho$ of $\Sigma$ different from $\rho_i$, has dimension at most one.
	
	The line $\Lambda_{ab}$ corresponding to a point $\varphi(v)+ae_{01}+be_{23} \in \varphi(\rho_i)+\sigma$ is depicted in Figure~\ref{fig:lineab}. To see this just consider the two points $p_1=\varphi(v)+(a+1,a,0,0)$ and $p_2=\varphi(v)+(0,0,b+1,b)$ on it, fix a realization $L$ of $\Lambda_{ab}$ and take two points $x_1, x_2$ with valuation $p_1$ and $p_2$. The coordinate of $\Lambda_{ab}$ corresponding to the basis element $e_{ij}$ can now be computed as the valuation of $x_iy_j-x_jy_i$. This equals $\min\{(p_1)_i+(p_2)_j, (p_2)_i+(p_1)_j\}$ because $(p_1)_i+(p_2)_j \neq (p_2)_i+(p_1)_j$ for every $i,j$, and so one can compute that $\Lambda_{ab}$ corresponds to the point $\varphi(v)+ae_{01}+be_{23}$.
	
	We denote by $\delta$ the union of all the tropicalized lines $\Lambda_{ab}$ for $a,b\geq 0$, $\delta$ is the union of two cones pointed at $v$, $\delta= (v + \op{pos}( e_0, e_1 )) \cup (v + \op{pos}( e_2, e_3 ))$. We have that $\delta$ is contained in the union of the affine planes $(v+ H_{01}) \cup (v + H_{23})$. 
	As we assumed that $v\notin H_{01}  \cup H_{23}$ the two affine planes $v+ H_{01}$ and $v+H_{23}$ do not pass through the origin. As a result any ray $\rho_j$ intersects $\delta$ in at most two points because the line it spans intersects the two affine planes $v+H_{01}$ and $v+H_{23}$ that contain $\delta$ in at most one point each.
	
	This gives us a finite set $S \subset \delta$ made of the intersection points of all the rays $\rho \neq \rho_i$ with $\delta$. A line $\Lambda_{ab}$ intersects some ray $\rho$ if and only if it contains a point $p\in S$, and we just have to prove that, for each $p\in S$, the set of points $\varphi(v)+ae_{01}+be_{23}$, with $0\leq a<\alpha$, $0\leq b<\beta$ and such that the line $\Lambda_{ab}$ contains $p$, has dimension at most one.
	
	Let $p\in \delta$ be a point not contained in the line $v+\op{span}(1,1,0,0)$. We can assume, without loss of generality, that $p \in v+ H_{01}$ so that $p= v+c_0e_0+c_1e_1$ for some $c_0, c_1 \geq 0$. A tropical line $\Lambda_{ab}$ contains $p$ if and only if $a= \min\{c_0, c_1\}$. It follows that the set of points $\varphi(v)+ae_{01}+be_{23}$, with $0\leq a< \alpha$ and $0\leq b<\beta$ and such that the line $\Lambda_{ab}$ contains the point $p$, is the line segment $\{\varphi(v)+\min\{c_0,c_1\}e_{01}+be_{23} \mid 0\leq b < \beta\}$.
	
	It remains to prove that the points $p \in S$ contained in the line $v+\op{span}(1,1,0,0)$ are only contained in finitely many tropicalized lines $\Lambda_{ab}$ with $a\leq \alpha$ and $b\leq \beta$ for some $\alpha$ and $\beta$. We can finally define $\alpha$ and $\beta$: $\alpha$ is the minimum positive $t$ such that there is a point $p\in S$ of the form $p= v - t (1,1,0,0)$, and $\beta$ the minimum positive $t$ such that there is a point $p\in l$ of the form $p=v + t (1,1,0,0)$. By definition no line $\Lambda_{ab}$ with $a<\alpha$ and $b < \beta$  contain any point $p \in S$ contained in $v+\op{span}(1,1,0,0)$, and this concludes the proof.
\end{proof}

\begin{remark}\label{os:setinj}
The proof of Theorem~\ref{curveinj} also shows that if $Z_{\Sigma}$ and $Z_{\Sigma'}$ have the same support set, then the support set of $\Sigma$ and $\Sigma'$ is the same outside $\Lambda$. This is consistent with Example~\ref{eg:contro}.
\end{remark}

\section{Tropicalization of families} 
We now give an application of tropical Chow hypersurfaces to the study of tropicalization of families of algebraic varieties. In this section we will work with trivial valuation.

Consider a family of projective algebraic curves $\mathcal{X}\subset \mathbb{A}^k\times\mathbb{P}^3$ and denote, for $a\in \mathbb{A}^k$, its fiber by $X_a$. 
We will show how to use the theory of tropical Chow hypersurfaces to answer the following question.
\begin{question}\label{realqu}
What are all the possible tropical varieties $\Trop(X_a)$ as $a$ varies in $\mathbb{A}^k$?
\end{question}

We can associate to $\mathcal{X}$ a lift of the Chow form ${\op{ch}}_\mathcal{X}\in K[c_1, \ldots, c_k, p_I \mid I\in \binom{[4]}{2}]$. This form can be computed, for example, with the algorithm described in \cite[3.1]{DalStur}. 
 It has the property that, for any $a\in \mathbb{A}^k$ such that $X_a$ is a curve, a lift of the Chow form $\op{ch}_{X_a}$ can be obtained by substituting the $c_i$'s with the $a_i$'s in ${\op{ch}}_\mathcal{X}$. By a slight abuse of notation we will again denote this lift as $\op{ch}_{X_a}$.

\begin{proposition}\label{realprop}
Let $\mathcal{X}\subset \mathbb{A}^k\times\mathbb{P}^3$ be a family of projective algebraic curves and let $a,b \in \mathbb{A}^k$. Suppose that $\mathcal{N}^1(\op{Newt}(\op{ch}_{X_a}))$ and $\mathcal{N}^1(\op{Newt}(\op{ch}_{X_b}))$ intersect transversely the tropical Grassmannian $\op{TrGr}(1,3)$. Then  $\Trop(X_a) = \Trop(X_b)$ if and only if $\mathcal{N}^1(\op{Newt}(\op{ch}_{X_a})) \cap \op{TrGr}(1,3) = \mathcal{N}^1(\op{Newt}(\op{ch}_{X_b})) \cap \op{TrGr}(1,3)$. In particular if $\op{Newt}(\op{ch}_{X_a}) = \op{Newt}(\op{ch}_{X_b})$ then $\Trop(X_a) = \Trop(X_b)$.
\end{proposition}
\begin{proof}
The Chow hypersurface $Z_{X_a}$ is the intersection of $\V(\op{ch}_{X_a})$ with the Grassmannian $\op{Gr}(1,3)$. The tropicalization of $\V(\op{ch}_{X_a})$ is the codimension-one skeleton $\mathcal{N}^1(\op{Newt}(\op{ch}_{X_a}))$ of the normal fan of the Newton polytope of $\op{ch}_{X_a}$. As this intersection is transverse, we have that $\Trop(Z_{X_a}) = \op{TrGr}(1,3) \cap \Trop(\V(\op{ch}_{X_a}))$. The same argument shows that $\Trop(Z_{X_b}) = \op{TrGr}(1,3) \cap \Trop(\V(\op{ch}_{X_b}))$. The statement now follows from Theorem~\ref{curveinj}.
\end{proof}

\begin{remark}
The assumption on the dimension of the varieties $X_a$ is only due to the same assumption being made in Theorem~\ref{curveinj}. A proof of Conjecture~\ref{coinj} would automatically extend Proposition~\ref{realprop} to families of projective varieties of arbitrary dimension.
\end{remark}

Consider the stratification of $\mathbb{A}^k$ defined by the coefficients of the lift of the Chow form ${\op{ch}}_\mathcal{X}\in K[c_1, \ldots, c_k, p_I \mid I\in \binom{[4]}{2}]$. This stratification has as closed strata the vanishing loci of subsets of coefficients of ${\op{ch}}_\mathcal{X}$. Two points $a,b$ in the same stratum have the same Newton polytope $\op{Newt}(\op{ch}_{X_a}) = \op{Newt}(\op{ch}_{X_b})$ and hence, if the transversality condition of Proposition~\ref{realprop} holds, we also have $\Trop(X_a) = \Trop(X_b)$. The transversality condition holds in many cases, making this argument an useful tool to approach Question~\ref{realqu}. This is shown in the following example.

\begin{example}
Let $\tilde{I}\subset K[c_1,c_2][x,y,z,t]$ be the ideal generated by the polynomials
\begin{equation*} 
	f=x^2+y^2+zt,\ g=c_1z^2+c_2zt+xy+t^2.
\end{equation*}
The ideal $\tilde{I}$ defines a family of quartic curves in $\mathbb{P}^3$ parametrized by $\mathbb{A}^2$. The Chow form of $\tilde {I}$ is a polynomial of degree four in the Pl\"ucker coordinates with $47$ monomials:
\tiny  
\begin{multline*}
{\op{ch}}_{\mathcal{X}}=-c_1p_{01}^4+c_2p_{01}^2p_{02}^2-p_{02}^4+p_{01}^2p_{02}p_{12}+c_2p_{01}^2p_{1,2}^2-2p_{02}^2p_{1,2}^2-p_{1,2}^4 -4c_1p_{01}^2p_{02}p_{03}+2c_2p_{02}^3p_{03}\\
+c_1c_2p_{01}^2p_{03}^2+(-c_2^2-2c_1)p_{02}^2p_{03}^2+2c_1c_2p_{02}p_{03}^3-c_1^2p_{03}^4+2p_{02}^3p_{13}-4c_1p_{01}^2p_{1,2}p_{13}+4c_2p_{02}^2p_{1,2}p_{13}\\
+2p_{02}p_{1,2}^2p_{13}+2c_2p_{1,2}^3p_{13}+c_1p_{01}^2p_{03}p_{13}-2c_2p_{02}^2p_{03}p_{13}+2c_1p_{02}p_{03}^2p_{13}+c_1c_2p_{01}^2p_{13}^2+(-2c_2^2-4c_1-1)p_{02}^2p_{13}^2\\
-2c_2p_{02}p_{1,2}p_{13}^2+(-c_2^2-2c_1)p_{1,2}^2p_{13}^2+4c_1c_2p_{02}p_{03}p_{13}^2-2c_1^2p_{03}^2p_{13}^2+2c_1p_{02}p_{13}^3+2c_1c_2p_{1,2}p_{13}^3-c_1^2p_{13}^4\\
-3p_{01}p_{02}^2p_{23}-2c_2p_{01}p_{02}p_{1,2}p_{23}+p_{01}p_{1,2}^2p_{23}+c_2p_{01}p_{02}p_{03}p_{23}+c_1p_{01}p_{03}^2p_{23}+(2c_2^2+4c_1+1)p_{01}p_{02}p_{13}p_{23}\\
+c_2p_{01}p_{1,2}p_{13}p_{23}-2c_1c_2p_{01}p_{03}p_{13}p_{23}-3c_1p_{01}p_{13}^2p_{23}+(-c_2^2+2c_1)p_{01}^2p_{23}^2-4p_{02}p_{1,2}p_{23}^2+p_{02}p_{03}p_{23}^2\\
+4c_2p_{02}p_{13}p_{23}^2+p_{1,2}p_{13}p_{23}^2-4c_1p_{03}p_{13}p_{23}^2-2c_2p_{01}p_{23}^3-p_{23}^4.
\end{multline*}
\normalsize
The coefficients of ${\op{ch}}_{\mathcal{X}}$ are products of the following polynomials:
\begin{equation*}
c_1,\ c_2,\ c_2^2-2c_1,\ c_2^2+2c_1,\ 2c_2^2+4c_1+1.
\end{equation*}
This defines a stratification of $\mathbb{A}^2$ in seven strata: in each stratum the Newton polytope of the Chow form is constant. The closure of these strata are:
\begin{equation*}
\begin{array}{c}
\V_0=\mathbb{A}^2, V_1=\V(c_1),\ V_2=\V(c_2),\ V_3=\V(c_2^2-2c_1),\ V_4=\V(c_2^2+2c_1),\\ 
V_5=\V(2c_2^2+4c_1+1), \ V_{6}=\V(c_1, c_2),\ V_{7}=\V(c_1, 2c_2^2+4c_1+1 ),\\
 V_{8}=\V(c_2, 2c_2^2+4c_1+1),\ V_{9}=\V(c_2^2-2c_1,2c_2^2+4c_1+1),\ V_{10}=\emptyset.
\end{array}
\end{equation*}
For the strata $V_0, V_2,V_3,V_4, V_5,V_{7},V_{8}, V_{9},V_{10}$ the corresponding Newton polytope of ${\op{ch}}_{X_a}$ defines a tropical hypersurfaces in $\mathbb{R}^6/\mathbb{R}$ that is transverse to the Grassmannian $\op{TrGr}(1,3)$. As a result, by Proposition~\ref{realprop}, the tropicalization of $\V(\tilde{I})$ is constant within these strata.
The tropicalization within the stratum $V_1$ is constant too, as, for $c_1=0$ and $c_2\neq 0$ the tropical hypersurfaces $\Trop(\V(f))$ and $\Trop(\V(g))$ intersect transversely, and thus determine the tropical curve $\Trop(\V(f,g))$. Finally the tropicalization is trivially constant within the stratum $V_{6}=\{(0,0)\}$.

As a result these ten strata correspond to ten (non-empty) potentially different tropical varieties 
\begin{equation*}
\Sigma_0,\ \Sigma_1,\ \Sigma_2,\ \Sigma_3,\ \Sigma_4,\ \Sigma_5,\ \Sigma_{6},\ \Sigma_{7},\ \Sigma_{8},\ \Sigma_{9}.
\end{equation*}
The tropical varieties can be computed, for example with \texttt{gfan} (\cite{gfan}), as $\Trop(X_a)$ for arbitrary parameters $a$ in the correct stratum. We have that
\begin{multline*}
\Sigma_0=\{\op{pos}(1,0,0,0),\  \op{pos}(0,1,0,0),\ \op{pos}(0,0,1,0),\ \op{pos}(0,0,0,1)\},\\ \text{ with multiplicities } \{4,4,4,4\}, 
\end{multline*}
\begin{multline*}
\Sigma_1=\{\op{pos}(1,0,0,0),\  \op{pos}(0,1,0,0),\ \op{pos}(0,0,-1,1),\ \op{pos}(0,0,1,0)\},\\ \text{ with multiplicities } \{2,2,2,4\},
\end{multline*}
\begin{multline*}
\Sigma_{6}=\{\op{pos}(3,-1,-3,1),\ \op{pos}(0,0,1,0),\ \op{pos}(-1,3,-3,1)\},\\ \text{ with multiplicities } \{1,4,1\}. 
\end{multline*} 
Moreover 
\begin{equation*}
\Sigma_0 = \Sigma_2 = \Sigma_3 = \Sigma_4 = \Sigma_5 = \Sigma_{8} = \Sigma_{9} \text{ and }\ \Sigma_1=\Sigma_{7},
\end{equation*} 
so that there are actually only three possible tropical varieties arising as $\Trop(X_a)$ for some $a\in \mathbb{C}^2$.
All those identifications, with the exception of $\Sigma_1 = \Sigma_{7}$, are already visible looking at the Newton polytope of ${\op{ch}}_{X_a}$. This happens because each strata correspond to a different set of exponents of the specialized polynomial ${\op{ch}}_{X_a}$, but to a same Newton polytope as they have the same convex hull. 
\end{example}

\addcontentsline{toc}{section}{References}
\bibliographystyle{plain}
\bibliography{rrbib}

\end{document}